\documentclass[a4paper,11pt]{elsarticle}
\usepackage{fullpage,microtype}
\usepackage{amsmath,amsthm,amsfonts}
\usepackage{breqn}
\usepackage{thmtools}
\usepackage{thm-restate}
\usepackage{hyperref}
\usepackage{cleveref}
\usepackage{stmaryrd}

\newtheorem{proposition}{Proposition}
\newtheorem{corollary}{Corollary}

\let\ab\allowbreak

\begin{document}

\begin{frontmatter}
\title{Modulus of continuity eigenvalue bounds for homogeneous graphs and convex subgraphs with applications to quantum Hamiltonians}

\author[1,3]{Michael Jarret\corref{cor1}}
\ead{mjarret@pitp.ca}

\author[2,3]{Stephen P. Jordan}
\ead{stephen.jordan@nist.gov}

\address[1]{Department of Physics, University of Maryland, College Park, MD 20742}
\address[2]{Applied and Computational Mathematics Division, National Institute of Standards and Technology, Gaithersburg, MD 20899}
\address[3]{Joint Center for Quantum Information and Computer Science (QuICS), University of Maryland, College Park, MD 20742}

\cortext[cor1]{Corresponding author}

\begin{abstract}
 We adapt modulus of continuity estimates to the study of spectra of combinatorial graph Laplacians, as well as the Dirichlet spectra of certain weighted Laplacians. The latter case is equivalent to stoquastic Hamiltonians and is of current interest in both condensed matter physics and quantum computing. In particular, we introduce a new technique which bounds the spectral gap of such Laplacians (Hamiltonians) by studying the limiting behavior of the oscillations of their solutions when introduced into the heat equation. Our approach is based on recent advances in the PDE literature, which include a proof of the fundamental gap theorem by Andrews and Clutterbuck.
\end{abstract}
\end{frontmatter}

\section{Introduction} 
In this paper, we investigate the spectral structure of combinatorial graph Laplacians by adapting recent advances in the spectral theory of Schr\"odinger operators on $\mathbb{R}^n$. A combinatorial Laplacian $L$ corresponding to a connected graph $G$ of $N$ vertices has eigenvalues $0<\lambda_1(L)\leq \lambda_2(L) \leq \dots \leq \lambda_{N-1}(L)$ and corresponding eigenvectors $u_0,u_1,u_2,\dots,u_{N-1}$. In part of what follows, we focus on the spectral gap of $L$, or the difference in its two lowest eigenvalues. In this case, because $L$ always has lowest eigenvalue $0$, the spectral gap is simply $\lambda_1(L)$. 

To proceed, we introduce a technique based largely on the work of Ben Andrews, Julie Clutterbuck, and collaborators \cite{Andrews2011, andrews2012gradient, andrews2014moduli, andrews2013time, andrews2009lipschitz, andrews2013sharp}. Additionally, we attempt an approach similar to \cite{Andrews2011} to bounding the spectral gap $\gamma(H)$ of the physically-motivated case of a Hermitian matrix $H = L + W$, where $W$ is some diagonal matrix. Recently, these matrices have been called ``stoquastic Hamiltonians" in the physics literature \cite{Bravyi2006a}.\footnote{The term ``stoquastic" comes from the resemblance to stochastic matrices. Up to normalization, stoquastic matrices are equivalent to sub-stochastic matrices (\emph{cf.} \cite{Feller1956,Feller1957}). The spectral properties of sub-stochastic matrices have been previously studied in \cite{Lawler1988}. In graph theory, the current setting, these Hamiltonians correspond to Laplacians of subgraphs of weighted graphs with Dirichlet boundary, as discussed in \Cref{sec:Dirichlet} and elaborated on in \cite{Chung}.} That the lowest eigenvalue of $H$ is no longer $0$ and the corresponding eigenvector is nonuniform makes determining the spectral gap of $H$ a more challenging problem than that of $L$ alone. In this paper, we reduce such a bound to an estimate involving the log-concavity of the lowest eigenvector $u_0$ of $H$. 

Because this is the first attempt at applying these techniques to graph spectra, we simplify our problem by considering only \textit{homogeneous} graphs and their strongly convex subgraphs. A homogeneous graph $G$ has an associated group $\mathcal{H}$ and the edges associated with any vertex of $G$ may be identified with the elements of a particular generating set $\mathcal{K}$ for $\mathcal{H}$. (For a formal definition, see \Cref{ssec:IHG}.) We consider only the case such that $g^{-1} \in \mathcal{K}$ if and only if $g \in \mathcal{K}$ and therefore the graph is undirected. Also, we assume that the graph is \textit{invariant}, or that the generating set $\mathcal{K}$ is invariant under conjugation by elements $g \in \mathcal{K}$. A subgraph $S \subseteq G$ with vertex set $V(S)$ is \textit{strongly convex} if for each pair of vertices $x,y \in V(S)$, all of the shortest paths in $G$ from $x$ to $y$ are also contained in $S$ \cite{Chung}. 

Our approach follows \cite{Andrews2011}, where the authors proved the Fundamental Gap Conjecture. In particular, we study the behavior of oscillations in functions defined on the graph $V(S)$. In \cite{Andrews2011}, the authors studied the time-extended behavior of these oscillation terms when introduced into the heat equation, since such terms cannot decay any slower than $C e^{-\lambda_1(L)t}$ for some constant $C$. These oscillation terms are characterized by a modulus of continuity, a construct which typically tracks how uniformly continuous a function is, but we can think of as quantifying the size of oscillations separated by a particular distance. More specifically, for a function $f:V(S)\longrightarrow \mathbb{R}$ we say that it has modulus of continuity $\eta$ if
\begin{equation*}
	\lvert f(y)-f(x) \rvert  \leq \eta(d(y,x)) \; \text{for all $y,x \in V(S)$} 
\end{equation*}
where $d(y,x)$ is the shortest path length between vertices $y,x \in V(S)$. We will further formalize this modulus in \Cref{sec:homogeneous1}.

By sacrificing some tightness, one can apply modulus of continuity estimates without utilizing the heat equation at all. Instead one can derive bounds in terms of the $\ell^2$-norm of the modulus. Nonetheless, our intuition stems from the heat equation and we expect that the heat equation will prove useful in subsequent work, so we derive our results from this perspective.

In \Cref{sec:homogeneous1}, we prove the primary result of this paper: 

\begin{restatable}{thm}{graphbound}\label{thm:bound}
 Let $L$ be the combinatorial Laplacian for a strongly convex subgraph $S\subseteq G$ of an invariant homogeneous graph $G$. Then,
 \begin{equation*}
  \lambda_1(L) \geq 2 \left(1 - \cos\left(\frac{\pi}{D+1} \right)\right)
 \end{equation*}
 where $D$ is the diameter of $S$.
\end{restatable}

This theorem gives a nice lower bound to the spectral gap of combinatorial Laplacians in terms of the diameter of the corresponding graph. Although there is a long history of results comparing eigenvalues to diameters, this particular bound relates $\lambda_1(L)$ to the first eigenvalue of the path graph of $D+1$ vertices. This bound is also tight, since it is always achieved for $S \subset G$ such that $S$ is the path graph with $D$ edges. As a corollary to \Cref{thm:bound}, this bounds the eigenvalues of the normalized laplacian $\mathcal{L}$ of $S$. Thus, this provides a tight bound comparable to that of \cite{Chung1994}, where the author derives a lower bound of $1/(8k D^2)$ for the Neumann eigenvalues of $S$ where $k$ is the degree of $S$.

 In \Cref{sec:hypercube}, the proof strategy of \Cref{thm:bound} is adapted to the case of the hypercube graph. In particular, we recover the following, well-known bound:
 
\begin{restatable}{thm}{hypercube}\label{thm:hypercube}
	Let $L$ be the combinatorial Laplacian for a hypercube graph. Then, $\lambda_1(L) \geq 2$.
\end{restatable} 

 Since one can directly calculate that $\lambda_1(L)=2$ independently of $D$, this result is tight and demonstrates the power of modulus of continuity estimates adapted to spectral graph theory. In physical contexts, this estimate may also prove useful. We begin to explore such physical cases in \Cref{sec:Dirichlet}, where we consider matrices of the form $H=L+W$ where $W$ is any diagonal matrix and $L$ is a combinatorial Laplacian. For simplicity, we restrict $W$ to be positive-semidefinite, but since the spectral gap of $H$ is unaltered by an addition of a constant multiple of the identity matrix, our results apply equally well to all diagonal $W$. In particular, we derive the following bound on the spectral gap $\gamma(H) = \lambda_1(H)-\lambda_0(H)$:

\begin{restatable}{thm}{Dirichlet}\label{thm:Dirichlet}
	Let $(u_0, \lambda)$ and $(u_1, \lambda+\gamma)$ be the two lowest eigenvector-eigenvalue pairs of $H=L+W$ where $L$ is a combinatorial Laplacian of a strongly convex subgraph of an invariant homogeneous graph and $W$ is a diagonal positive-semidefinite matrix. Let the componentwise ratio $f=u_1/u_0$ have modulus of continuity $\eta$ and $g = \log(u_0)$. Then, 
	\begin{equation*}
		\gamma \geq 2 C_{u_0} \left(1-\cos\left(\frac{\pi}{D+1} \right)\right)
	\end{equation*}
	where $D$ is the diameter of $S$,
	\begin{equation*}
	C_{u_0} = \inf_{(y,x)\in \xi}\frac{\displaystyle \sum_{a\in\mathcal{K}}\Delta_a f(y)e^{g(ay)-g(y)} -\sum_{a\in\mathcal{K}}\Delta_a f(x)e^{g(ax)-g(x)}}{\displaystyle\sum_{a\in\mathcal{K}}\Delta_a f(y)-\sum_{a\in\mathcal{K}}\Delta_a f(x)},
	\end{equation*}	
	and 
	\begin{equation*}	
	\xi = \left\{\left(y,x\right) \in V(S) \; \vert \; \eta\left(\lvert y^{-1}x\rvert\right)= f(y)-f(x)\right\}.
	\end{equation*}
\end{restatable}  
Above, $\Delta_a f(x) = f(ax)-f(x)$ for $x \in V(S)$. This result reduces the task of bounding $\gamma(H)$ to determining an appropriate constant $C_{u_0}$ and is motivated similarly to the approach taken in \cite{Andrews2011}, where the authors prove the longstanding fundamental gap conjecture. 
 
 Extending results of the fundamental gap literature to discrete Laplacians was first considered by Ashbaugh and Benguria in \cite{Ashbaugh1990}, where the authors proved a fundamental gap-type theorem for the case of symmetric, single-well potentials on a one-dimensional Dirichlet Laplacian. More recently, in \cite{Jarret2014c} we proved another fundamental gap-type theorem for the case of convex potentials on one-dimensional combinatorial Laplacians and Hamming-symmetric convex potentials on hypercube combinatorial Laplacians by following the method of \cite{Lavine1994a}. In the context of hypercube combinatorial Laplacians $L$, we find in \Cref{sec:Dirichlet}:
 
\begin{restatable}{thm}{DirichletH}\label{thm:Dirichlet_Hypercube}
	Let $(u_0, \lambda)$ and $(u_1, \lambda+\gamma)$ be the two lowest eigenvector-eigenvalue pairs of $H=L+W$ where $L$ is the combinatorial Laplacian of a Hypercube graph $G$ and $W$ is some diagonal positive-semidefinite matrix. Let the componentwise ratio $f=u_1/u_0$ have modulus of continuity $\eta$. Let $g = \log(u_0)$. Then, $\gamma \geq 2 C_{u_0}$ with 
	\begin{equation*}
		C_{u_0} = \frac{\displaystyle \sum_{a\in\mathcal{K}}\Delta_a f(y)e^{g(ay)-g(y)} -\sum_{a\in\mathcal{K}}\Delta_a f(x)e^{g(ax)-g(x)}}{\displaystyle\sum_{a\in\mathcal{K}}\Delta_a f(y)-\sum_{a\in\mathcal{K}}\Delta_a f(x)}
	\end{equation*}
	for $y,x \in V(G)$ such that $f(y)-f(x) = \eta(2)$.
\end{restatable}
Here, $C_{u_0}$ is restricted to admit $y,x$ only if they are separated by at most a path of length 2. Hence, \Cref{thm:Dirichlet_Hypercube} presents a much more local property than \Cref{thm:Dirichlet}.

We also make use of a modulus of concavity $\omega$ of $\log(u_0)$ where $u_0$ is the ground-state of the operator $H$. By \textit{modulus of concavity}, we mean that for each pair of $y,x \in V(S)$ and some generator $a \in \mathcal{K}$ falling along a shortest path connecting $y$ to $x$, 
\begin{equation*}
	\Delta_a \log(u_0(y)) + \Delta_{a^{-1}}\log(u_0(x)) \geq \omega(d(y,x)) \; \text{for all $x,y\in V(S)$}.
\end{equation*}
 We apply the results of \Cref{sec:Dirichlet} to the case of path graphs with log-concave ground states to obtain the following bound:

\begin{restatable}{thm}{LC}\label{thm:LCbound}
	Suppose $H = L + W$ with ground state $u_0$, where $L$ is the combinatorial Laplacian for some path graph $S$ with diameter $D$ and $W:V(S)\longrightarrow \mathbb{R}_{\geq 0}$. Then,
	\begin{align*}
		\gamma(H) & \geq 4\left(2\cosh(\overline{\omega})-1\right)\left(1- \cos\left(\frac{\pi}{2D +1}\right)\right) \\
		& \geq 4 \left(1- \cos\left(\frac{\pi}{2D +1}\right)\right)
	\end{align*}	
	for $log(u_0)$ having non-negative modulus of concavity $\omega$ and $\overline{\omega}=\inf_s \omega(s)$.
\end{restatable}

We can actually apply a closer analysis in deriving \Cref{thm:LCbound}, assuming that we know a bound on the gradient of the modulus of concavity $\omega$:
\begin{restatable}{thm}{LCHarder}\label{thm:LCbound2}
	Suppose $H = L + W$ with ground state $u_0$, where $L$ is the combinatorial Laplacian for some path graph $S$ with diameter $D$ and $W:V(S)\longrightarrow \mathbb{R}_{\geq 0}$. Then,
	\begin{align*}
		\gamma(H) \geq 4\left(1-\cos\left(\frac{\pi}{2 D + 1}\right)\right) + 2\inf_s\left(\Delta^- \cosh(\omega(s))\right)
	\end{align*}	
	for $log(u_0)$ with non-negative modulus of concavity $\omega$ where $\omega(D+1)=0$ and $\overline{\omega}=\inf_s \omega(s)$. Above,
	\begin{equation*}
		\Delta^- \cosh(\omega(s)) =  \cosh(\omega(s)) - \cosh(\omega(s+1)).
	\end{equation*}
\end{restatable}

This equation is particularly useful if we choose the modulus of concavity of $\omega$ to be convex. Such a restriction is always possible without altering our analysis, because we are concerned with finite graphs, but these considerations will be discussed in future work. Under such restrictions, \Cref{thm:LCbound2} provides the bound
\begin{equation}\label{eqn:intro_bound}
	\gamma(H) \geq 4\left(1-\cos\left(\frac{\pi}{2 D + 1}\right)\right) + 2\left(\cosh(\overline{\omega})-1\right)
\end{equation}
with $\overline{\omega}$ defined as in \Cref{thm:LCbound}. It is easy to see that \Cref{thm:LCbound2} is indeed an improvement over \Cref{thm:LCbound}. 

\section{Preliminaries}
In this paper we restrict our attention to spectra of invariant, homogeneous graphs and their strongly convex subgraphs. We introduce some algebraic tools for discussing such graphs in \Cref{ssec:IHG}. In \Cref{ssec:SGT}, we introduce the \textit{combinatorial Laplacian} and properties of its spectra. 

\subsection{Invariant homogeneous graphs}\label{ssec:IHG}
Let $G = (V,E)$ be a graph with vertex set $V(G)$ and edge set $E(G)$. We call $G$ \textit{homogeneous} if there exists a group $\mathcal{H}$ acting on $G$ such that for $\{u,v\} \in E(G)$, $\{a u , a v\} \in E(G) \; \forall \; a \in \mathcal{H}$ and $\forall u',v' \in V(G)\; \exists \; a_0 \in \mathcal{H}$ such that $a_0 u' = v'$. We call the set $\mathcal{K} \subset \mathcal{H}$ the edge generating set if $a \in \mathcal{K} \iff \{v,av\} \in E(G) \; \forall \; v \in V(G)$. 

We restrict to the case that $G$ is undirected; hence, if $\{v,av\} \in E(G)$ we also have that $\{av, v\} \in E(G)$. This restriction is equivalent to requiring that $a \in \mathcal{K} \iff a^{-1} \in \mathcal{K}$.\footnote{Note that if $v = a v$ and $g(av) = v$, then $g = a^{-1}$.} To simplify our problem further, we reduce our class of graphs by insisting that these graphs be \textit{invariant} homogeneous graphs, or that $a \mathcal{K} a^{-1} = \mathcal{K} \; \forall \; a \in \mathcal{K}$.

We also need a notion of distance in the graph. Typically, we use $d(x,y)$, the length of the shortest path connecting vertex $x$ to vertex $y$. In our setting, it helps to formalize this in group-theoretic terms. Because we are considering invariant homogeneous graphs, we can take $d(x,y)=|w|$, where $|\cdot|$ represents the \textit{word metric} over $\mathcal{K}$ and $|w|$ is the length of the shortest word $w$ written in terms of elements of $\mathcal{K}$ such that $w x = y$. 

\begin{proposition}\label{prop:sp1}
	Let $G$ be an invariant homogeneous graph with generating set $\mathcal{K}$. Then, for $x,y \in V(G)$ and $a \in \mathcal{K}$, $d(ax,ay)=d(x,y)$.
\end{proposition}
\begin{proof}
	This follows immediately from the equivalence of the shortest path and the word metric. Begin by writing $w x = y$. Then, for some $a \in \mathcal{K}$, $a w a^{-1} (a x) = a y$. By invariance, $|a w a^{-1}| = |w|$ and we have that $d(ax,ay)=d(x,y)$. 
\end{proof}
Because \Cref{prop:sp1} demonstrates the proper equivalence between the word-metric measured in generators of $\mathcal{K}$ and the distance between vertices $y,x$, we will write $|y^{-1}x|$ to represent $d(y,x)$.

Now, let $S$ be an induced subgraph of $G$. We label the boundary of $S$ by $\delta S = \{v \in V(G) \setminus V(S) | v\sim u \in S \}$. $S$ is said to be \textit{strongly convex} if it satisfies the following two (equivalent) properties:

\begin{enumerate}
\item For all pairs of vertices $y,x \in S$, the shortest path connecting $y$ to $x$ is also in $S$.
\item For all $a,b \in \mathcal{K}$, $x \in \delta S$, if $ax \in S$ and $b x \in S$ then $b^{-1}a  \in \mathcal{K}$.\cite{Chung1994}
\end{enumerate}

\begin{proposition}\label{prop:sp2}
    Let $S \subseteq G$ be a strongly convex induced subgraph of an invariant homogeneous graph $G$. If $x,ax,y \in S$ and $d(ax,y) = d(x,y) + 1$, then $ay \in S$.
\end{proposition}

\begin{proof}
    Suppose that $x,y \in S$ and $d(ax,y) = d(x,y) + 1$. By \Cref{prop:sp1}, we know that $d(ax,ay)=d(x,y)$ and thus there exists a shortest path traversing $ax\rightarrow ay \rightarrow y$. Hence, $ay \in S$. 
\end{proof}

\subsection{Graph Laplacians}\label{ssec:SGT}

The focus of this paper is the \textit{combinatorial Laplacian} $L$ of a graph $S$ which for all $x,y \in S$ is given by
\begin{equation}\label{eq:Laplacian}
	L(x,y) = \begin{cases}
	d_x & \text{if $x=y$} \\
	-1 & \text{if $x\sim y$} \\
	0 & \text{otherwise}
	\end{cases}
\end{equation}
where $d_x$ is the degree of vertex $x$. $L$ can also be identified with an operator on the space of functions $u:V(S) \longrightarrow \mathbb{R}$ satisfying
\begin{equation}\label{eqn:opdef}
	Lu(x) = \sum_{y\sim x} \left( u(x) - u(y) \right).
\end{equation}
The reader should note that the operator $L$ in \cref{eqn:opdef} should be understood to apply to $u$ before $u$ is evaluated at the vertex $x$. In the case that $S$ is an induced subgraph of a homogeneous graph $G$ with edge generating set $\mathcal{K}$, we can equivalently write
\begin{equation}\label{eqn:LaplacianOp}
	Lu(x) = \sum_{a \in \mathcal{K}_x} \left(u(x)-u(ax) \right)
\end{equation}
where $\mathcal{K}_x = \left\{a \in \mathcal{K} \; \vert \; ax \notin \delta S\right\}$. Here $\mathcal{K}_x$ is simply the set that generates all vertices in $S$ adjacent to some particular vertex $x \in V(S)$. 

The operator $L$ corresponding to a connected graph has eigenvalues $\lambda_0(L) < \lambda_1(L) \leq$ $\dots$ $\leq \lambda_{|V(G)|-1}(L)$ and corresponding eigenvectors $u_0(L),u_1(L) ,$ $ \dots, $ $ u_{|V(G)|-1}(L)$ with $\lambda_0(L)=0$ and
\begin{equation}\label{eq:ev1}
	\lambda_1(L) = \inf_{u \perp \mathbf{1}} \frac{\displaystyle\sum_{x\sim y}\left(u(x)-u(y)\right)^2}{\displaystyle \sum_x u^2(x)}
\end{equation}
where $\mathbf{1}$ is the constant function. $u$ attaining the infimum in \cref{eq:ev1} is called a \textit{combinatorial harmonic eigenfunction} of $S$ and can be identified with an eigenvector of $L$. If $(u,\lambda)$ is an eigenvector-eigenvalue pair of $L$, then $u$ satisfies
\begin{equation}\label{eq:recurrence}
-\lambda u(x) = \sum_{y \sim x} \left( u(y) - u(x) \right).
\end{equation}
Although \cref{eq:recurrence} is the standard definition of an eigenvector and can be obtained by inspecting \cref{eqn:opdef}, the expression can also be derived through through variational techniques on \cref{eq:ev1}\cite{Chung}. 

\section{Main Results}
\subsection{Strongly convex subgraphs of invariant homogeneous graphs}\label{sec:homogeneous1}
Heat kernel techniques are one of the more powerful approaches to proving eigenvalue bounds \cite{Chung}. In this section, we adapt the approach of \cite{andrews2014moduli} to combinatorial Laplacians. This technique has the advantage of often being easier to handle than known techniques, such as those of \cite{Banuelos2000,Chung,Chung2014,Chung1994,Chung2000}, while often retaining (and potentially sharpening) these bounds. In particular, the results of this section are comparable to those of \cite{Chung1994}.

For a graph $G$ with diameter $D$, we begin by considering solutions to the initial value problem
\begin{equation}\label{eqn:heat}
    \begin{cases}
    	\frac{d(s,t)\phi}{dt} = - L \phi(s,t) \\
    	\phi(s,0) = \phi_0(s).
    \end{cases}
\end{equation}
Clearly, if we let $\phi_0=u$ for an eigenvector-eigenvalue pair $(u,\lambda)$ of $L$, we have that $\phi(s,t) = u(s) e^{-\lambda t}$ solves \cref{eqn:heat}. Our strategy, then, is to consider the decay rate of oscillations in $u$. Since $\lambda_0(L)=0$, the slowest such oscillations decay is proportional to $e^{-\lambda_1 t}$. Thus, if we bound the decay rate of these oscillations, we implicitly bound on the spectral gap. To characterize the magnitude of oscillations, we introduce the modulus of continuity for a function defined on a graph.

For a function $f:V(G)\times\mathbb{R}^+ \longrightarrow \mathbb{R}$, we call $\eta:[-D,D]\times\mathbb{R}^+\longrightarrow \mathbb{R}$ its \textit{modulus of continuity} if
\begin{equation}\label{eqn:modulus}
    \eta(s,t) = \begin{cases}
    			\sup_{y,x\in V(G)} \left\{ f(y,t)-f(x,t) \; \vert \; \lvert y^{-1}x\rvert \leq s \right\} & s >0 \\
    			0 & s=0 \\
    			-\sup_{y,x\in V(G)} \left\{ f(y,t)-f(x,t) \; \vert \; \lvert y^{-1}x\rvert \leq -s \right\} & s <0 
    			\end{cases}
\end{equation}
Although traditionally we would define the modulus only over non-negative $s$, defining it as an anti-symmetric function about the origin is advantageous for the analysis that follows. Importantly, our choice of $\eta$ is monotonic and sub-additive, which further simplifies many of the arguments that follow. In future settings, however, it may be worth utilizing alternatively restricted moduli, such as concave moduli. Since we are interested in finite graphs, there always exists a concave function that both lies above and touches $\eta$. In fact the analysis that follows applies to these moduli equally well, but would require more detail than is necessary in the current context.

To prove \Cref{thm:bound}, we need the following fact.
\begin{proposition}\label{prop:BCs}
 Suppose $S \subseteq G $ is a finite strongly convex subgraph in an invariant homogeneous graph $G$ with edge generating set $\mathcal{K}$. Let $u: V(S) \longrightarrow \mathbb{R}$ have modulus of continuity $\eta$. Then, for $y,x,ay \in V(S)$ either $ax \in V(S)$ or $\left\lvert u(ay)-u(x) \right\rvert \leq \eta\left(\lvert y^{-1}x \rvert \right)$.
 \end{proposition}
 \begin{proof}
 	To prove this, simply note that from \Cref{prop:sp2} we know that either $ax \in V(S)$ or $\lvert y^{-1}(ax) \rvert \leq \lvert y^{-1}x \rvert$. Thus, $ax \in V(S)$ or $u(ay)-u(x) \leq \eta(\lvert y^{-1} x \rvert)$.
 \end{proof}
 \begin{proposition}\label{prop:BCs2}
 Suppose $S \subseteq G $ is a finite strongly convex subgraph in an invariant homogeneous graph $G$ with edge generating set $\mathcal{K}$. Let $u: V(S) \longrightarrow \mathbb{R}$ have modulus of continuity $\eta$. Then, for $y,x \in V(S)$ achieving the supremum in $\eta(\lvert y^{-1}x\rvert)$ with $u(y)\geq u(x)$
	\begin{enumerate}
		\item \label{item:BCs2} if $ay \in V(S)$ and $ax \notin V(S)$, then $u(ay)-u(y) \leq 0$ and
		\item \label{item:BCs3} if $ax \in V(S)$ and $ay \notin V(S)$, then $u(ax)-u(x) \geq 0$.
	\end{enumerate}
\end{proposition}
\begin{proof}
 To prove \cref{item:BCs2}, assume that $ax \notin V(S)$ and write
 \begin{align*}
   	u(ay)-u(y) &= u(ay)-u(y) + u(x) - u(x) \\
   	&= u(ay) - u(x) -\eta(\lvert y^{-1} x \rvert) \\
   	&\leq \eta(\lvert y^{-1} x \rvert) - \eta(\lvert y^{-1} x \rvert) \\
   	&= 0
   \end{align*}
   where the inequality follows from \Cref{prop:BCs}. \Cref{item:BCs3} is similar to \cref{item:BCs2} and proof is omitted. 
\end{proof}

\begin{proposition}\label{prop:reduce}
Suppose $S$ is a finite strongly convex subgraph in an invariant homogeneous graph $G$ with edge generating set $\mathcal{K}$. Let $u: V(S) \longrightarrow \mathbb{R}$ have modulus of continuity $\eta$.  Then, for $y,x \in V(S)$ achieving the supremum in $\eta(\lvert y^{-1}x\rvert)$ with $u(y)\geq u(x)$
\begin{equation*}
	-Lu(y)+Lu(x) \leq \sum_{a \in \mathcal{Y}} \left( u(ay) - u(y) \right) - \sum_{a \in \mathcal{X}}\left( u(ax) - u(x) \right)
\end{equation*}
for any $\mathcal{Y}\subseteq \mathcal{K}_y$ and $\mathcal{X}\subseteq \mathcal{K}_x$ satisfying $\mathcal{Y}\cap(\mathcal{K}_y\cap\mathcal{K}_x) = \mathcal{X}\cap(\mathcal{K}_y\cap\mathcal{K}_x)$.
\end{proposition}
\begin{proof}
From \cref{eqn:LaplacianOp} we have,
	\begin{dgroup*}
		\begin{dmath*}
		-Lu(y)+Lu(x) = \sum_{a \in \mathcal{K}_{y}}\left(u(ay) - u(y)\right) - \sum_{a \in \mathcal{K}_{x}} \left(u(ax) - u(x)\right) 
		\end{dmath*}
		\begin{dmath}\label{eqn:zero}
		{} = \sum_{a \in \mathcal{K}_{y}\cap\mathcal{K}_x}\left(u(ay) - u(y)\right) - \sum_{a \in \mathcal{K}_{y}\cap\mathcal{K}_x} \left(u(ax) - u(x)\right) + \sum_{a \in \mathcal{K}_{y}\setminus\mathcal{K}_x}\left(u(ay) - u(y)\right) - \sum_{\mathcal{K}_x \setminus\mathcal{K}_y} \left(u(ax) - u(x)\right).
		\end{dmath}
	\end{dgroup*}

	Now, since $\mathcal{K}_{y}\setminus\mathcal{K}_x$ is the set of all $a \in \mathcal{K}$ such that $ay \in V(S)$ and $ax \notin V(S)$ and similarly for $\mathcal{K}_{x}\setminus\mathcal{K}_y$, from \Cref{prop:BCs2} we know that 
	\begin{dgroup*}
		\begin{dmath}\label{eqn:one}
			\sum_{a \in \mathcal{K}_{y}\setminus\mathcal{K}_x}\left(u(ay) - u(y)\right) \leq \sum_{a \in \mathcal{J}_y}\left(u(ay) - u(y)\right)
		\end{dmath}
		\begin{dmath}\label{eqn:two}
			-\sum_{a \in \mathcal{K}_{x}\setminus\mathcal{K}_y}\left(u(ax) - u(x)\right) \leq -\sum_{a \in \mathcal{J}_x}\left(u(ax) - u(x)\right)
		\end{dmath}
	\end{dgroup*}
	for any $\mathcal{J}_y\subseteq\mathcal{K}_{y}\setminus\mathcal{K}_x$ and $\mathcal{J}_x\subseteq\mathcal{K}_{x}\setminus\mathcal{K}_y$. Now, noting that $u(y)-u(x) = \eta(\lvert y^{-1}x \rvert)$, \Cref{prop:sp1} implies that $u(ay)-u(ax) \leq u(y)-u(x)$ for any $a\in \mathcal{K}_y\cap\mathcal{K}_x$. Thus,
	\begin{dmath}\label{eqn:three}
	\sum_{a \in \mathcal{K}_{y}\cap\mathcal{K}_x}\left(u(ay) - u(y)\right) - \sum_{a \in \mathcal{K}_{y}\cap\mathcal{K}_x} \left(u(ax) - u(x)\right) \leq \sum_{a \in \mathcal{J}}\left(u(ay) - u(y)\right) - \sum_{a \in \mathcal{J}} \left(u(ax) - u(x)\right)
	\end{dmath}
	for any $\mathcal{J}\subseteq\mathcal{K}_y\cap\mathcal{K}_x$. Combining \cref{eqn:zero,eqn:one,eqn:two,eqn:three} completes the proof.
\end{proof}

\begin{proposition}\label{prop:MOCheat}
 Suppose $S$ is a finite strongly convex subgraph with even (odd) diameter $D$ of an invariant homogeneous graph $G$ with edge generating set $\mathcal{K}$ and $|\mathcal{K}|=k$. If $u:V(S)\times \mathbb{R}^+ \longrightarrow \mathbb{R}$ is a solution of \cref{eqn:heat}, then the modulus of continuity $\eta$ of $u$ satisfies for positive even (odd) $s$,
 \begin{equation}
	 	\frac{d\eta(s,t)}{dt} \leq - L_P \eta(s,t)
 \end{equation} 
 where $L_P$ is the combinatorial Laplacian of the path graph $P$ with $V(P)=\{s \; \vert \; \ab s \in \llbracket -D, D \rrbracket \; \ab \text{and \ab $s$ even (odd)} \}$ and $E(P) = \ab \{\{s,s+2\} \; \vert \; \ab s\in \llbracket -D, D-2 \rrbracket \;\text{and $s$ even (odd)} \}$.
\end{proposition}

\begin{proof}
 Choose $y,x$ to achieve the supremum in \cref{eqn:modulus} with $u(y) \geq u(x)$. Say that $s=\lvert y^{-1} x \rvert \in \mathbb{E}$ where $\mathbb{E}$ is the appropriate choice of the set of all evens or all odds. Then, we have that
    \begin{equation}\label{eqn:12}
        \frac{d \eta(s,t)}{dt}\bigg\rvert_{t=t_0} = \left(\frac{d u(y,t)}{dt}-\frac{d u(x,t)}{dt}\right)\bigg\rvert_{t=t_0} = -Lu(y,t_0)+Lu(x,t_0)
    \end{equation}
    where, to avoid excessive notation, we have adopted the convention $u(y)=u(y,t_0)$. Now, fix $a_0,a_1 \in \mathcal{K}$ such that $a_0 y$ and $a_1 x$ lie along a shortest path connecting $y$ to $x$. 
By our choice, we know that $a_0 y \in S$ and $a_1 x \in S$. Now, we have a few cases and in each we will apply \Cref{prop:reduce} with various choices of $\mathcal{Y},\mathcal{X}$. (In the cases that follow, we adopt the convention that $\eta(D+2)=\eta(D+1)=\eta(D)$.) 
\paragraph{Case 1, $a_0x \in S$ and $a_1 y \in S$:} In this case, we choose $\mathcal{Y}=\mathcal{X}=\{a_0,a_1\}$. Hence, \Cref{prop:reduce} and \cref{eqn:12} yield
    \begin{align*}
        \frac{d \eta(s,t)}{dt} &\leq u(a_0y) + u(a_1y)-2u(y) - u(a_0x) -u(a_1x)+2u(x) \\
        &= \left(u(a_0y)-u(a_1x)\right) + \left(u(a_1y)-u(a_0x)\right) - 2 \left(u(y)-u(x)\right) \\
        &\leq \eta(s-2) + \eta(s+2) -2 \eta(s) \\
        &= - L_P\eta(s) 
    \end{align*}
    where the final inequality follows from \cref{eqn:modulus}.
\paragraph{Case 2, $a_0x \notin S$ and $a_1y \in S$:} In this case, we choose $\mathcal{Y}=\{a_0,a_1\}$ and $\mathcal{X}=\{a_1\}$. Hence, \Cref{prop:reduce} and \cref{eqn:12} yield
	\begin{align*}
        \frac{d \eta(s)}{dt} &\leq u(a_0y) + u(a_1y)-2u(y) - u(a_1x) + u(x) \\
        &= \left(u(a_0y)-u(a_1x)\right) + \left(u(a_1y)-u(x)\right) - 2 \left(u(y)-u(x)\right) \\
        &\leq \eta(s-2) + \eta(s+1) -2 \eta(s) \\
        &\leq \eta(s-2) + \eta(s+2) - 2 \eta(s)\\
        &= - L_P\eta(s)
    \end{align*}
    where the inequalities follow from \cref{eqn:modulus}.
\paragraph{Case 3, $a_0x \in S$ and $a_1y \notin S$:} This is similar to Case 2 and proof is omitted.
\paragraph{Case 4, $a_0x \notin S$ and $a_1y \notin S$:} In this case, we choose $\mathcal{Y} = \{a_0\}$ and $\mathcal{X}=\{a_1\}$. Hence, \Cref{prop:reduce} and \cref{eqn:12} yield
	\begin{align*}
        \frac{d \eta(s)}{dt} &\leq u(a_0y)-u(y) - u(a_1x) + u(x) \\
        &= \left(u(a_0y)-u(a_1x)\right) - \left(u(y)-u(x)\right) \\
        &\leq \eta(s-2) -\eta(s) \\
        &\leq \eta(s-2) + \eta(s+2) - 2 \eta(s)\\
        &= - L_P\eta(s) 
    \end{align*}
    where the inequalities follow from \cref{eqn:modulus}.
    
    Thus, in all cases, 
    \begin{equation*}
    \frac{d \eta(s)}{dt} \leq -L_P \eta(s)
    \end{equation*}
	provided $s\geq 0$.
\end{proof}

\graphbound*
\begin{proof}
    Let $\lambda_1 = \lambda_1(L)$. Suppose $u_1$ is a solution to \cref{eqn:heat} and $(u_1,\lambda_1)$ is the first eigenvector-eigenvalue pair of $L$. Let $\eta$ be the modulus of continuity for $u_1$. For simplicitly, we restrict our attention to $\eta(s)$ such that $s\in \mathbb{E}$ in accordance with \Cref{prop:MOCheat}. In other words, we treat $\eta$ as a vector with entries indexed by $s \in \mathbb{E}$. Then, \Cref{prop:MOCheat} yields
    \begin{equation*}
        \frac{d\eta(s)}{dt} \leq - L_P  \eta(s) \;\; \text{for $s \geq 0$}.
    \end{equation*}
    Noting that $\eta$ as defined in \Cref{prop:MOCheat} is an odd function, we immediately see that
    \begin{equation*}
        \frac{d\eta(s)}{dt} \geq - L_P  \eta(s) \;\; \text{for $s < 0$}
    \end{equation*}
    so that we have
    \begin{equation*}
        \eta^{\top}\frac{d\eta}{dt} \leq - \eta^{\top}L_P  \eta.
    \end{equation*}
    Then,
    \begin{align*}
    		\frac{1}{2}\frac{d \lvert \eta \rvert^2}{dt} &\leq -\eta^{\top} L_P \eta \\
    		&\leq - \mu \lvert \eta \rvert^2
    \end{align*}
    where $\mu= 2\left(1-\cos\left(\frac{\pi}{D+1}\right)\right)$ is the smallest non-trivial eigenvalue of $L_P$. Hence, we have that 
    \begin{equation*}
        \lvert \eta(s) \rvert \leq C e^{-\mu t}
    \end{equation*}
    for $s \in \llbracket -D,D \rrbracket$ and some constant $C$ chosen independently of $t$. Then, there exist $y,x \in V(S)$ such that,
    \begin{align*}
        \lvert u_1(y,0)-u_1(x,0)\rvert e^{-\lambda_1 t} & = \eta(s,t) \\
        &\leq C e^{-\mu t} \\
        \lvert u_1(y,0)-u_1(x,0) \rvert &\leq C e^{(\lambda_1-\mu)t}.
    \end{align*}
    Note that $u_1(y,0)-u_1(x,0)$ is nonzero, so that if $\lambda_1-\mu < 0$ we arrive at a contradiction by taking $t \rightarrow \infty$. Hence, $\lambda_1 \geq \mu$.
\end{proof}

We can alternatively prove \Cref{thm:bound} without using the heat equation:
\begin{proof}
	Let $\lambda_1 = \lambda_1(L)$. Suppose that $(u_1,\lambda_1)$ is the first eigenvector-eigenvalue pair of $L$. Let $u_1$ have modulus of continuity $\eta$, and let vertices $x$ and $y$ achieve the supremum defining $\eta(s)$ in \cref{eqn:modulus}. Then
	\begin{equation*}
	-L u_1(y) + L u_1(x) = \lambda_1 \eta(s).
	\end{equation*}
	As shown in the proof of \Cref{prop:MOCheat},
	\begin{equation*}
		-L u_1(y) + Lu_1(x) \leq -L_P \eta(s).
	\end{equation*}
	Hence,
	\begin{equation*}
	-\lambda_1 \eta(s) \leq -L_P \eta(s).
	\end{equation*}
	Now, since $\eta(s) > 0$ for all $s>0$, we have that for $s>0$,
	\begin{equation*}
		-\lambda_1 \eta^2(s) \leq -\eta(s)L_P\eta(s).
	\end{equation*}
	Recalling that $\eta$ is odd, this yields
	\begin{align*}
		-\lambda_1 \lvert \eta \rvert^2 &\leq - \eta^{\top} L_P \eta \\
		&\leq - \mu \lvert \eta \rvert^2
	\end{align*}
	where $\mu= 2\left(1-\cos\left(\frac{\pi}{D+1}\right)\right)$ is the smallest non-trivial eigenvalue of $L_P$. Thus, since $\lvert \eta \rvert^2$ is nonzero, $\lambda_1 \geq \mu$ and we have proven \Cref{thm:bound}.
\end{proof}

One should note that the two proofs of \Cref{thm:bound} are essentially the same, as the lower bound on the decay-rate of the heat equation can be deduced from the $\ell^2$-norm of the modulus. Regardless, while the first method can be adapted to any non-constant function $u_1$, the latter cannot. Also, the reader familiar with normalized Laplacians should note that as a consequence of \Cref{thm:bound}, we obtain a lower bound of $\frac{2}{k}\left(1-\cos\left(\frac{\pi}{D+1}\right)\right)$ on the spectral gap of the normalized Laplacian for convex subgraphs of homogeneous graphs, where $k$ is the degree of the graph. Thus, we can compare this result to those of \cite{Chung1994,Chung}.

\subsection{Example 1: Path graphs}
Consider any path graph and note that it is a convex subgraph of some homogeneous graph. Then, \Cref{thm:bound} implies that the first eigenvalue 
\begin{equation*}
\lambda_1(L) \geq 2\left(1-\cos\left(\frac{\pi}{D+1}\right)\right).
\end{equation*}
 This bound is tight, since the eigenvalues of the path graph are actually given by 
 \begin{equation*}
 \lambda_j(L) = 2 \left(1-\cos\left(\frac{j \pi}{D+1}\right)\right).
 \end{equation*}

\subsection{Example 2: Hypercube graphs}\label{sec:hypercube}
\hypercube*
\begin{proof}
For the hypercube, we can choose $\mathcal{K}$ such that it is both abelian and every element $a \in \mathcal{K}$ is self-inverse. We again consider a solution $u$ to \cref{eqn:12} with modulus of continuity $\eta$. Then, $\eta$ either satisfies $\eta(2)>\eta(1)$ or $\eta(2)=\eta(1)$. Let $y,x$ be the vertices that achieve the supremum in $\eta(2)$ with $u(y)\geq u(x)$.

\paragraph{Case 1, $\lvert y^{-1}x\rvert=2$:} Note that in this case we can write $y = b' b x$ for some $b,b' \in \mathcal{K}$ and that $y\neq x$ implies $b \neq b'$. Then, \Cref{eqn:12} with $s =2$ becomes
\begin{align*}
 \frac{d\eta(2)}{dt} &= \sum_{a \in \mathcal{K}} \left(u(ay)-u(y) \right) - \sum_{a\in \mathcal{K}} \left(u(ax)-u(x) \right) \\
 &\leq (u(by)-u(bx)) + (u(b'y)-u(b'x) -2\eta(2)) \\
 &= (u(b'y) - u(bx)) + (u(by) - u(b'x)) - 2 \eta(2) \\
 & = -2 \eta(2).
\end{align*}
Above, the first inequality follows from \Cref{prop:reduce} with $\mathcal{Y}=\mathcal{X}=\{b,b'\}$. 

\paragraph{Case 2, $\lvert y^{-1}x\rvert=1$:} \Cref{eqn:12} with $s=1$ becomes 
\begin{align*}
 \frac{d\eta(1)}{dt} &= \sum_a \left(u(ay)-u(y) \right) - \sum_a \left(u(ax)-u(x) \right) \\
 &\leq (u(by)-u(bx)) + (u(b'y)-u(b'x) -2\eta(1)) \\
 &= (u(b'y) - u(b'x)) + (u(x) - u(y)) - 2 \eta(1) \\
 & \leq -2 \eta(1)
\end{align*}
where the first inequality follows from \Cref{prop:reduce} with $\mathcal{X}=\mathcal{Y}=\{b,b'\}$ with $b$ satisfying $x=by$ and $bx=y$. The second inequality follows from the definition of $\eta$. Thus, in either case we have that
\begin{equation}\label{eqn:hypercube_bound}
\frac{d\eta(2)}{dt} \leq -2 \eta(2).
\end{equation}
Now, by either method of \Cref{thm:bound}, $\lambda_1(L) \geq 2$ and our bound is tight.
\end{proof}
It is both remarkable and (perhaps) expected that the particular connectivity of the hypercube allows us to consider only points separated by a path of length 2 while still obtaining a tight bound. The modulus of continuity approach suggests that in many cases of physical interest the spectral gap is a highly local property. This result may be exploitable in the context of quantum Ising models, where it can reduce our problem to that of estimating the log-concavity of the ground-state wavefunction (the lowest eigenvector). 

\section{Dirichlet Eigenvalues and Ising-type Hamiltonians}\label{sec:Dirichlet}
Now we consider the more general problem of bounding the gap of the matrix $H=L + W$, where $L$ is the combinatorial Laplacian for some subgraph $S$ of a homogeneous graph and $W$ is a positive-semidefinite matrix. In the physics literature these are known as ``stoquastic Hamiltonians" and have the same spectrum as the Dirichlet eigenvalues of $S$ for an appropriate choice of host graph. The key results of \Cref{sec:multidimensional} should be seen as \Cref{prop:ratio_rel} and \Cref{cor:hamiltonian}. 

The constant $C_{u_0}$ introduced in \Cref{thm:Dirichlet} and \Cref{thm:Dirichlet_Hypercube} requires further exploration before it provides useful bounds. However, we believe that in the case that $u_0$ is log-concave, for some suitably-defined notion of log-concavity, $C_{u_0}\geq 1$. \Cref{sec:log-concave} applies the techniques of \Cref{sec:multidimensional} to derive a bound on the spectral gap of $H$ in the one-dimensional case. \Cref{thm:LCbound} and \Cref{thm:LCbound2} should be viewed as a slightly weakened (but still strong) analogue of \Cref{thm:Dirichlet}, demonstrating the utility of the methods of \cref{sec:multidimensional} and the promise of an alternative expression for \Cref{thm:LCbound} and \Cref{thm:LCbound2} entirely in terms of (a measure of) the log-concavity of $u_0$ and the diameter of $S$.

\subsection{Induced subgraphs of weighted homogeneous graphs and Hamiltonians with potentials}\label{sec:multidimensional}
In this section, we consider an induced subgraph $S$ of a graph $G$ with vertex set $V(S) \subseteq V(G)$ and nonempty vertex boundary $\delta S$. We let $S' = \{\{x,y\}\in E(G) \; \vert \; x \in V(S) \; \text{or} \; y \in V(S)\}$. In other words, $S'$ is the set of all edges with at least one end in $S$. Then, we define the lowest (combinatorial) Dirichlet eigenvalue of the induced subgraph $S$ as
\begin{equation}\label{eqn:Dirichlet_ground}
	\lambda^{(D)}_0 = \inf_{u \in D^{*}}\frac{\displaystyle\sum_{\{x,y\} \in S'}\left(u(x)-u(y)\right)^2}{\displaystyle\sum_{y \in V(S)} u^2(y)}
\end{equation}
where $D^*$ is simply the set of all nonzero functions satisfying the Dirichlet condition
\begin{equation*}
	u(x) = 0 \; \text{for} \; x \in \delta S.
\end{equation*}
The function $u_0:V(S)\cup \delta S \rightarrow \mathbb{R}$ achieving the infimum in \cref{eqn:Dirichlet_ground} is called a Dirichlet eigenfunction and in accordance with the physics literature, we refer to $u_0$ as a ground-state. In the interior of $S$, $u_0$ is nonzero and has constant sign, so is taken to be completely positive. Hence, there exists a function $g:V(S)\cup\delta S \longrightarrow \mathbb{R}$ satisfying 
\begin{equation}\label{eqn:log-state}
	u_0(y) = 
	\begin{cases}
		e^{g(y)} & y \in V(S) \\
		0 & y \in \delta S.
	\end{cases}
\end{equation}
Above, $g$, the log of the ground-state, will prove a more natural consideration in much of what follows. In general, any function $u_0:V(S)\cup \delta S \rightarrow \mathbb{R}$ such that $u_0 > 0$ interior to $S$ is compatible with such a choice of $g$. To specify a function consistent with \cref{eqn:log-state} for some $u_0 >0$, we will often write $g=\log(u_0)$. 

Higher Dirichlet eigenvalues can be defined generally by
\begin{equation*}
	\lambda^{(D)}_i = \inf_{\substack{u \perp C_i \\ u \in D^{*}}}\frac{\displaystyle\sum_{\{x,y\} \in S'}\left(u(x)-u(y)\right)^2}{\displaystyle\sum_{y \in V(S)} u^2(y)}
\end{equation*}
where $C_i$ is the subspace spanned by the $i$ lowest nonzero Dirichlet eigenfunctions.\footnote{Note that $\lambda^{(D)}_i$ differ from those of the corresponding normalized Laplacian only by a factor of $k$, the degree of $G$.} Imposing the Dirichlet condition explicitly, we can write
\begin{equation}\label{eqn:Dirichlet_ev}
	\lambda^{(D)}_{i} = \inf_{u \perp C_i}\frac{\displaystyle\sum_{x \sim y \in S}\left(u(x)-u(y)\right)^2 + \sum_{y \in V(S)} W(y) u^2(y)}{\displaystyle\sum_{y \in V(S)} u^2(y)}
\end{equation}
where $W(y)=\lvert \{\{x,y\}\in S' \; \vert \; x \in \delta S \}\rvert$. Thus, we identify $\lambda^{(D)}_i$ with the eigenvalues of the matrix $L+W$ where $L$ is the combinatorial Laplacian of $S$ and $W$ is some diagonal matrix with non-negative integer-valued entries. 

For the remainder of this section, we adopt a somewhat more general construction. We let $H = L + W$ where $W$ is any positive-semidefinite diagonal matrix. Equivalently, $W:V(G) \longrightarrow \ab \mathbb{R}_{\geq 0}$. (Since we are ultimately concerned with spectral gaps, we could equivalently discuss $W$ as any diagonal matrix by simply shifting $W\mapsto c I + W$ for any $c$ without impacting the spectral gap.) Despite relaxing the combinatorial constraints on $W$, the eigenvalues of $H$ are still given by \cref{eqn:Dirichlet_ev}.\footnote{These are also the Dirichlet eigenvalues for the weighted combinatorial Laplacian with $\displaystyle W(u) = \sum_{\{v,u\} \in \partial S} w(v,u)$ and unit weight on the edges internal to $S$.} $H$ defined this way corresponds to a subset of so-called ``stoquastic Hamiltonians'' which have been of recent interest in quantum theory \cite{Bravyi2006a}.\footnote{One important stoquastic Hamiltonian would be the transverse-field Ising model with a non-negative field.} Since solutions of \cref{eqn:Dirichlet_ev} are simply the eigenvalues of $H$, for the remainder of this section we write $\lambda_{i} = \lambda^{(D)}_i$.

To bound the spectral gap $\gamma(H)$, we once again wish to consider solutions to the heat equation
\begin{equation}\label{eqn:heat3}
    \begin{cases}
    	\frac{d\phi(s,t)}{dt} = - H \phi(s,t) \\
    	\phi(s,0) = \phi_0(s).
    \end{cases}
\end{equation}
In general, we proceed by consider the componentwise ratio of two solutions $u_0$ and $u_1$ to \cref{eqn:heat3}, where we choose $u_0 > 0$ in the interior of $S$. This situation is rather similar to that considered in \Cref{sec:homogeneous1}, but we require a relationship like \cref{eqn:LaplacianOp} to proceed. To this end, we propose the following:

\begin{proposition}\label{prop:ratio_rel}
 Let $S$ with combinatorial Laplacian $L$ be a convex induced subgraph of some invariant homogeneous graph. Let $u_0(x,t),u_1(x,t)$ be solutions to \cref{eqn:heat3} with $u_0(x,0)=u_0(x)$ and $u_1(x,0)=u_1(x)$ and satisfying the Dirichlet condition on $\delta S$. Suppose $u_0(x) > 0$ for all $x \in V(S)$. If
	\begin{equation*}
		f(x,t) = \frac{u_1(x,t)}{u_0(x,t)} \; \text{for $x \in V(S)$}. 
	\end{equation*} 
 and 
	\begin{equation*}
		\Delta_a f(x,t) = \begin{cases}
			f(ax,t)-f(x,t) & ax \in V(s) \\
			0 & ax \in \delta S.
		\end{cases}
	\end{equation*} 
Then,
 \begin{equation*}
  \frac{d f}{dt} = \sum_{a \in \mathcal{K}} \Delta_a f(x,t) e^{g(ax)-g(x)}
 \end{equation*}
 for $g = log(u)$ defined consistently with \cref{eqn:log-state}.
\end{proposition}
\begin{proof}
For simplicity, we write $f=f(t)$ and similarly for $u_0(t),u_1(t)$. Then,
\begin{dgroup*}
	\begin{dmath*}
		\frac{d f(x)}{dt} = \frac{1}{u_0(x)} \frac{d u_1(x)}{dt} - \frac{u_1(x)}{u^2_0(x)} \frac{d u_0(x)}{dt}.
	\end{dmath*}
\end{dgroup*}
	If $u_1(x)=0$, the second term is $0$ and the remainder of the proof becomes trivial. Hence, we assume that $u_1(x)\neq 0$. Now, we recall that $H=L+W$ where $W$ is diagonal and apply \cref{eqn:heat3} to get
\begin{dgroup*}
	\begin{dmath*}
		\frac{d f(x)}{dt} = f(x) \left(\frac{1}{u_1(x)}\frac{d u_1(x)}{dt} - \frac{1}{u_0(x)} \frac{d u_0(x)}{dt} \right)
	\end{dmath*}
	\begin{dmath*}
		= f(x) \left(-\frac{1}{u_1(x)}H u_1(x) + \frac{1}{u_0(x)}H u_0(x)\right)
	\end{dmath*}
	\begin{dmath*}
		= f(x) \left(-\frac{1}{u_1(x)}L u_1(x) + \frac{1}{u_0(x)}L u_0(x)  -\frac{1}{u_1(x)} W u_1(x) + \frac{1}{u_0(x)}W u_0(x)\right)
	\end{dmath*}
	\begin{dmath*}
		= f(x) \left(-\frac{1}{u_1(x)}L u_1(x) + \frac{1}{u_0(x)}L u_0(x)  -W(x) + W(x) \right)
	\end{dmath*}
	\begin{dmath*}
		= f(x) \left(\frac{1}{u_1(x)}\sum_{a \in \mathcal{K}}\left(u_1(ax)-u_1(x)\right) - \frac{1}{u_0(x)}\sum_{a \in \mathcal{K}}\left(u_0(ax)-u_0(x)\right)\right)
	\end{dmath*}
	\begin{dmath*}
		= f(x) \left(\sum_{a \in \mathcal{K}}\left(\frac{u_1(ax)}{u_1(x)}-1\right) - \sum_{a \in \mathcal{K}}\left(\frac{u_0(ax)}{u_0(x)}-1\right)\right)
	\end{dmath*}
	\begin{dmath*}
		= f(x)\sum_{a \in \mathcal{K}}\left(\frac{u_1(ax)}{u_1(x)} - \frac{u_0(ax)}{u_0(x)}\right)
	\end{dmath*}
	\begin{dmath*}
		= f(x)\sum_{\substack{a \in \mathcal{K} \\ ax \in V(S)}}\left(\frac{f(ax)}{f(x)}-1\right)\frac{u_0(ax)}{u_0(x)}
	\end{dmath*}
	\begin{dmath*}
		= \sum_{\substack{a \in \mathcal{K} \\ ax \in V(S)}}\left(f(ax)-f(x)\right)\frac{u_0(ax)}{u_0(x)}
	\end{dmath*}
	\begin{dmath*}
		= \sum_{a \in \mathcal{K}} \Delta_a f(x) \frac{u_0(ax)}{u_0(x)}
	\end{dmath*}
	\begin{dmath*}
		= \sum_{a \in \mathcal{K}} \Delta_a f(x) e^{g(ax) - g(x)}.
	\end{dmath*}
\end{dgroup*}
\end{proof}

\begin{corollary}\label{cor:hamiltonian}
 Let $(u_0, \lambda_0),(u_1,\lambda_0+\gamma)$ be the first and second eigenvector-eigenvalue pair of $H=L+W$ where $L$ is a combinatorial Laplacian and $W$ is a diagonal positive-semidefinite matrix. Then, if $u_0(t),u_1(t)$ are solutions to \cref{eqn:heat3} with $u_0(0)=u_0$ and $u_1(0)=u_1$, we have that
 \begin{equation*}
  -\gamma f(x) = \sum_{a\in \mathcal{K}} \Delta_a f(x) e^{g(ax,t)-g(x,t)}.
 \end{equation*}
 for $g$ defined consistently with \cref{eqn:log-state}.
\end{corollary}
\begin{proof}
 This follows from \Cref{prop:ratio_rel} by simply noting that $f(x,t)=\frac{u_1(x)}{u_0(x)}e^{-\gamma t}$.
\end{proof}

%


Note that the operator acting on $f$ and satisfying the relationships of \Cref{prop:ratio_rel} and \Cref{cor:hamiltonian} has a constant eigenfunction with eigenvalue $0$. Thus, the analysis of \Cref{sec:homogeneous1} carries over identically, provided that we can appropriately bound $-\gamma f(x)$. Because of this, \Cref{prop:ratio_rel} and \Cref{cor:hamiltonian} are sufficient to prove \Cref{thm:Dirichlet}.

\Dirichlet*
\begin{proof}
First, let $\eta$ be the modulus of $f$. Then, note that by \Cref{prop:reduce}, for all $(y,x) \in \xi$ 
\begin{equation*}
	\sum_{a\in\mathcal{K}}\Delta_a f(y)-\sum_{a\in\mathcal{K}}\Delta_a f(x) < 0
\end{equation*}
for appropriate choice of $\mathcal{Y},\mathcal{X}$. Additionally, by the method of \Cref{prop:MOCheat}
\begin{equation*}
	\sum_{a\in\mathcal{K}}\Delta_a f(y)-\sum_{a\in\mathcal{K}}\Delta_a f(x) \leq -L_P \eta(\lvert y^{-1} x \rvert )
\end{equation*}
with $L_P$ defined as in \Cref{thm:bound}. 

Further, \Cref{cor:hamiltonian} requires that
\begin{equation*}
\sum_{a\in\mathcal{K}}\Delta_a f(y)e^{g(ay)-g(y)}-\sum_{a\in\mathcal{K}}\Delta_a f(x) e^{g(ax)-g(x)} < 0.
\end{equation*}
Thus, 
\begin{align*}
	C_{u_0} &= \inf_{\{y,x\}\in \xi}\frac{\displaystyle \sum_{a\in\mathcal{K}}\Delta_a f(y)e^{g(ay)-g(y)} -\sum_{a\in\mathcal{K}}\Delta_a f(x) e^{g(ax)-g(x)}}{\displaystyle\sum_{a\in\mathcal{K}}\Delta_a f(y)-\sum_{a\in\mathcal{K}}\Delta_a f(x)} >0
\end{align*}
Now, we apply \Cref{prop:ratio_rel} and obtain
\begin{align*}
    \frac{d \eta(s)}{dt} &= \sum_{a\in\mathcal{K}}\Delta_a f(y)e^{g(ay)-g(y)}-\sum_{a\in\mathcal{K}}\Delta_a f(x)e^{g(ax)-g(x)} \\
    &\leq C_{u_0} \left(\sum_{a\in\mathcal{K}}\Delta_a f(y)-\sum_{a\in\mathcal{K}}\Delta_a f(x)\right) \\
    &\leq -C_{u_0} L_P \eta(\lvert y^{-1}x \rvert).
\end{align*}
Hence, by the exact same argument as \Cref{thm:bound}, we have that $\gamma \geq C_{u_0} \mu$. Thus,
\begin{equation*}
		\gamma \geq 2 C_{u_0} \left(1-\cos\left(\frac{\pi}{D+1} \right)\right).
\end{equation*}
\end{proof} 

\DirichletH*

Proof of \Cref{thm:Dirichlet_Hypercube} is omitted, since it exactly follows the approach to \Cref{thm:Dirichlet}.

\subsection{Example 3: Log-concave ground states}\label{sec:log-concave}
In this section we apply the techniques above to prove a gap bound in the case that $H=L+W$ has a log-concave ground state $u_0$ for $L$ corresponding to a one-dimensional graph $S$. In particular, by log-concavity we mean that $g:V(S)\cup \delta S \longrightarrow \mathbb{R}$ defined consistently with \cref{eqn:log-state} satisfies
\begin{equation}\label{eqn:concave}
	\sum_{a \in \mathcal{K}} \left(g(ay) - g(y)\right) \leq 0 \; \text{for all $y \in V(S)$}.
\end{equation}
In more general settings, this is not a satisfactory notion of concavity, since the analogue of a saddle-point might also satisfy this definition. However, for the one-dimensional case considered in this section, it is appropriate. In the future, we will likely define a much stronger notion of concavity that acts as a better analogue to the continuous definition while still being useful in our setting. Regardless, note that concavity as defined by \cref{eqn:concave} can be trivially satisfied by $g$ at any vertex connected to the boundary $\delta S$. To see this, simply note our freedom in $g$ in \cref{eqn:log-state} and choose $g(ay) \rightarrow -\infty$ for any $ay \in \delta S$.

In the case of the path graph $S$, we choose our edge generating set $\mathcal{K}=\left\{b,b^{-1}\right\}$ and log-concavity implies that $g(bx)-2g(x)+g(b^{-1}x) \leq 0$ for all $x \in V(S)$. 

We also introduce in this section a modulus of concavity for $g$. This modulus allows us to prove tighter bounds than the simple assumption of log-concavity itself. For a graph $S$ with diameter $D$, we call $\omega:[0,D]\longrightarrow \mathbb{R}$ the modulus of concavity of a function $g$ defined on $V(S)$ if

\begin{equation}\label{eqn:modulus_contraction}
    \omega(s) = \inf_{\lvert y^{-1}x\rvert = s} \left\{ \frac{\Delta_{a^{-1}} g(y) + \Delta_{a}g(x)}{2} \; \bigg\vert \; \lvert y^{-1}a^{2}x\rvert \leq \lvert y^{-1}x\rvert\right\}.
\end{equation}

Basically, the modulus of concavity tells us exactly how strongly concave $g$ is over a particular path separation $s$. Its utility lies in the expectation that as the ground-state becomes more contracted, the spectral gap should increase.

\begin{proposition}\label{prop:pathop}
	Suppose $S$ is a path graph of diameter $D$ and $f:V(S)\times \mathbb{R} \longrightarrow \mathbb{R}$ and $\gamma$ are defined as in \Cref{cor:hamiltonian}. Let $\eta$ be the modulus of continuity of $f$. Then, for $s \geq 1$, $\eta$ satisfies
	 \begin{dgroup*}
	 	\begin{dmath*}
	 	-\gamma \eta(s) \leq -2L_P\eta(s) - 4 (\cosh(\omega(s))-1) \nabla \eta(s) 
	 	\end{dmath*}
	 \end{dgroup*}
	 where $\omega$ is the modulus of concavity of the ground state of $S$, $L_P$ is the combinatorial Laplacian operator for the path graph $P$ with $V(P)=\llbracket -D,D\rrbracket$ and $E(P)=\{\{s,s+1\}\}_{s\in\llbracket -D,D-1\rrbracket}$, and $\nabla$ is the operator defined by
	 \begin{equation*}
	 	\nabla \eta(s) = \eta(s)-\eta(s-1).
	 \end{equation*}
\end{proposition}

\begin{proof}
	 First, we let $f$ and $g$ be defined as in \Cref{prop:ratio_rel} with $f$ having modulus of continuity $\eta$. For simplicity, let $f(\cdot)=f(\cdot,t)$. Then, for $y,x$ achieving $\eta(s)$ with $f(y,t) > f(x,t)$, we have that
	 \begin{equation*}
	 	-\gamma \left(f(y)-f(x)\right) = \sum_{a\in \mathcal{K}} \Delta_a f(y) e^{g(ay)-g(y)} -\sum_{a\in \mathcal{K}} \Delta_a f(x) e^{g(ax)-g(x)}.
	 \end{equation*}
	 
	 Suppose that $b^{-1}y$ and $bx$ lie along a shortest path connecting $y$ to $x$. We begin by considering the interior terms
	 \begin{equation*}
	 \Psi_i \equiv \Delta_{b^{-1}} f(y) e^{g(b^{-1}y)-g(y)} - \Delta_{b} f(x)e^{g(bx)-g(x)}.
	 \end{equation*}
	 In particular,
	 \begin{align*}
	 	\Delta_{b^{-1}}f(y) &= f(b^{-1}y)-f(y) \\
	 	&= f(b^{-1}y)-f(y) + f(x) - f(x) \\
	 	&= f(b^{-1}y)-f(x)-\eta(s) \\
	 	&\leq \eta(s-1)-\eta(s).
	 \end{align*}
	 In similar fashion, we also have that $-\Delta_b f(x) \leq \left(\eta(s-1)-\eta(s)\right)$. Hence,
	 \begin{align*}
	 	\Psi_i &\leq \left(\eta(s-1)-\eta(s)\right)\left(e^{\Delta_b^{-1}g(y)} + e^{\Delta_b g(x)} \right) \\
	 	&= \left(\eta(s-1)-\eta(s)\right)\exp\left(\frac{\Delta_{b^{-1}}g(y)+\Delta_b g(x)}{2}\right)\left(e^{p}+e^{-p}\right)
	 \end{align*}
	 where $p= \frac{\Delta_{b^{-1}}g(y)-\Delta_b g(x)}{2}$. Then,
	 \begin{align*}
	 	\Psi_i &\leq 2 \cosh(p)\left(\eta(s-1)-\eta(s)\right)\exp\left(\frac{\Delta_{b^{-1}}g(y)+\Delta_b g(x)}{2}\right) \\
	 	&\leq 2 \cosh(p) \left(\eta(s-1)-\eta(s)\right) e^{\omega(s)}.
	 \end{align*}
	 where the final inequality comes from the definition of $\omega$ and the fact that $\eta(s-1)\leq \eta(s)$. The outer terms follow a similar procedure, where $\Delta_b f(y) \leq \eta(s+1)-\eta(s)$ and $-\Delta_{b^{-1}}f(x) \leq \eta(s+1)-\eta(s)$. For these, we have that
	 \begin{align*}
	 	 \Psi_o &\equiv \Delta_{b} f(y)e^{\Delta_b g(y)} - \Delta_{b^{-1}} f(x)e^{\Delta_{b^{-1}}g(x)} \\
	 	 &\leq \left(\eta(s+1)-\eta(s)\right)\left(e^{\Delta_b g(y)} + e^{\Delta_{b^{-1}}g(x)}\right) \\
	 	 &\leq \left(\eta(s+1)-\eta(s)\right)\left(e^{-\Delta_{b^{-1}} g(y)} + e^{-\Delta_{b}g(x)}\right) \\
	 	 &= 2 \cosh(p)\left(\eta(s+1)-\eta(s)\right)\exp\left(\frac{-\Delta_{b^{-1}} g(y)-\Delta_{b}g(x)}{2}\right) \\
	 	 &\leq 2 \cosh(p)\left(\eta(s+1)-\eta(s)\right)e^{-\omega(s)}.
	 \end{align*}
	 Above, the second inequality follows from log-concavity and the final inequality follows from the definition of $\omega$.
	 
	 Combining $\Psi_i$ and $\Psi_o$ we have that,
	 \begin{align*}
	 	-\gamma \left(f(y)-f(x)\right) &= \Psi_i + \Psi_o \\
	 	&\leq 2\cosh(p) \left( \left(\eta(s-1)-\eta(s)\right) e^{\omega(s)} + 2 \left(\eta(s+1)-\eta(s)\right)e^{-\omega(s)} \right)\\
	 	&\leq 2\cosh(p)\left(- L_P \eta(s) + R \right)
	 \end{align*} 
	 where
	 \begin{align*}
	 	R &\equiv \left(\eta(s-1)-\eta(s)\right) (e^{\omega(s)}-1) + \left(\eta(s+1)-\eta(s)\right)(e^{-\omega(s)}-1).
	 \end{align*}
	 Above, because log-concavity requires that $\omega(s)\geq 0$ and $\eta$ is monotonic, both terms in $R$ are independently non-positive. Thus, 
	 \begin{align*}
	 	R &= \left(\eta(s-1)-\eta(s)\right) (e^{\omega(s)}-1) + \left(\eta(s+1)-\eta(s)\right)(e^{-\omega(s)}-1) \\
	 	&= 2 \left(\eta(s-1)-\eta(s)\right)(\cosh(\omega(s))-1) + \left(\eta(s+1)-\eta(s-1)\right)(e^{-\omega(s)}-1) \\
	 	&\leq 2 \left(\eta(s-1)-\eta(s)\right)(\cosh(\omega(s))-1)
	 \end{align*}
	 and we arrive at 
	 \begin{dgroup*}
	 	\begin{dmath*}
	 	-\gamma \eta(s) \leq 2\cosh(p) \left(-L_P \eta(s) - 2 (\cosh(\omega(s))-1)\nabla\eta(s)\right).
	 	\end{dmath*}
	 \end{dgroup*}
	 Since the above inequality is trivially satisfied (and hence the proposition proven) if $-L_P \eta(s) - 2 (\cosh(\omega(s))-1)\nabla\eta(s) \geq 0$, we note that $\cosh(p)\geq 1$ and then 
	 \begin{dgroup*}
	 	\begin{dmath*}
	 	-\gamma \eta(s) \leq -2L_P \eta(s) - 4 (\cosh(\omega(s))-1)\nabla\eta(s).
	 	\end{dmath*}
	 \end{dgroup*}
\end{proof}	

We now use \Cref{prop:pathop} to perform various estimates on the spectral gap $\gamma(H)$. For our first estimate:
\LC*
\begin{proof}
	We begin with \Cref{prop:pathop},
	\begin{align*}
	 	-\gamma \eta(s) &\leq -2L_P\eta(s) - 4 \cosh(\omega(s))-1))\nabla\eta(s)	\\
	 	&= -2L_P\eta(s) - 4 (\cosh(\omega(s))-1)\left(\eta(s)-\eta(s-1)\right) \\
	 	{} &\leq -2L_P\eta(s) - 4 (\cosh(\omega(s))-1)\left(2\eta(s)-\eta(s+1)-\eta(s-1)\right)\\
	 	{} &= -2L_P\eta(s) - 4 (\cosh(\omega(s))-1)L_P \eta(s)\\
	 	{} &= -2L_P\eta(s)(2\cosh(\omega(s))-1))
	 \end{align*}
	 where $L_P$ is defined as in \Cref{prop:pathop} and the only inequality comes from adding a multiple of the non-negative term $\eta(s+1)-\eta(s)$. Hence, by the same analysis as \Cref{thm:bound},
	 \begin{equation*}
		 \gamma(H) \geq 4 \inf_s \left(2\cosh(\omega(s))-1\right)\left(1-\cos\left(\frac{\pi}{2D+1}\right)\right).
	\end{equation*}
\end{proof}

Although this proof follows immediately from \Cref{prop:pathop}, taking $\omega \rightarrow 0$ and comparing to \Cref{thm:bound} reveals that it is not tight. For one, the methods of \Cref{prop:pathop} are loose when $\omega(s)\sim 0$. This case is, of course, better handled by an approximation using the techniques of \Cref{sec:homogeneous1}. Nonetheless, we can still improve upon the estimate of \Cref{thm:LCbound} in the case that the gradient of $\omega$ is bounded. 
\LCHarder*
\begin{proof}
	We once again begin with the result of \Cref{prop:pathop}
	\begin{dgroup*}
	 	\begin{dmath*}
	 	-\gamma \eta(s) \leq -2L_P \eta(s) - 4 (\cosh(\omega(s))-1) \nabla \eta(s),
	 	\end{dmath*}
	 \end{dgroup*}
	 and look to estimate the contribution of the term associated with the operator $\nabla$. To do so, we consider the expected value of the associated term under $\eta$. Now, let $\omega(D+1)=0$ and then, if $\nabla'\eta(s)=\cosh(\omega(s))-1)\nabla\eta(s)$,
	 \begin{align*}
	 	\eta^{\top} \nabla' \eta &= \sum_{s=1}^{D} \eta(s)\left(\eta(s)-\eta(s-1)\right)(\cosh(\omega(s))-1) \\
	 	&\geq \sum_{s=1}^{D} \frac{\eta(s)+\eta(s-1)}{2}\left(\eta(s)-\eta(s-1)\right)(\cosh(\omega(s))-1) \\
	 	&=\frac{1}{2}\sum_{s=1}^{D}\left(\eta^2(s)-\eta^2(s-1)\right)(\cosh(\omega(s))-1)
	\end{align*}
	where the inequality follows from the monotonicity of $\eta$. Then,
	\begin{align*}	 	
	 	2 \eta^{\top} \nabla' \eta &=\sum_{s=1}^{D}\eta^2(s)(\cosh(\omega(s))-1) - \sum_{s=1}^{D}\eta^2(s-1)(\cosh(\omega(s))-1) \\
	 	&= \sum_{s=1}^{D}\eta^2(s)(\cosh(\omega(s))-1) - \sum_{s=1}^{D-1}\eta^2(s)(\cosh(\omega(s+1))-1) \\
	 	&= \sum_{s=1}^{D-1}\eta^2(s)(\cosh(\omega(s))-\cosh(\omega(s+1))) + \eta^2(D)\left(\cosh(\omega(D))-1)\right) \\
	 	&= \sum_{s=1}^{D}\eta^2(s)\Delta^- \cosh(\omega(s)) \\
	 	&\geq \inf_s\left(\Delta^- \cosh(\omega(s))\right)\sum_{s=1}^{D}\eta^2(s).
	 \end{align*}
	  Hence,
	\begin{equation*}
		\frac{2 \eta^{\top} \nabla' \eta}{\lvert \eta \rvert^2} \geq \inf_s \Delta^- \cosh \left( \omega(s)\right).
	\end{equation*}	 
	 
	 Thus, our estimate from \Cref{thm:LCbound} can be improved to
		\begin{align*}
			\gamma(H) &\geq 4\left(1-\cos\left(\frac{\pi}{2 D + 1}\right)\right) + 2\inf_s\left(\Delta^- \cosh(\omega(s))\right).
		\end{align*}
\end{proof}	

\section{Discussion and Future Work}
In general, modulus of continuity methods seem readily adaptable to both spectral graph theory and quantum theory. In particular, the results of \Cref{sec:homogeneous1} demonstrate that these estimates are quite strong for at least a certain class of graphs. The results of \Cref{sec:Dirichlet} are not immediately applicable in physical contexts, however \Cref{sec:log-concave} demonstrates ways in which they might be applied. These results can be strengthened by learning more about the relationship between the ratio $u_1/u_0$ and $u_0$ itself. Additionally, although a weak restriction, log-concavity may be an overly strong characterization of $u_0$ for practical purposes and one may prefer to derive results entirely in terms of the modulus of concavity of $\log(u_0)$. Further, bounds on the modulus of concavity of $\log(u_0)$ should be reducible to bounds on the modulus of concavity of the potential term $W'$ as seen in \cite{Andrews2011}. This comparison theorem is saved for future work, but since the potential term $W'$ is typically provided in both physical and quantum-computational contexts, in common settings this modulus of concavity should be explicitly calculable.

To advance these methods, we need to reduce the higher-dimensional cases of \Cref{sec:multidimensional} to the one-dimensional case of \Cref{sec:log-concave}. The results of \citep{Andrews2011} suggest that this is indeed possible, however proof in the graph-theoretic setting remains elusive. Such a theorem will likely follow from a stronger definition of concavity, so that we can make more direct comparisons of the weights $e^{\Delta_a g(y)}$. Although this looks promising, appropriately controlling the inequalities in each term of the sums of \Cref{prop:ratio_rel} and \Cref{thm:Dirichlet} appears difficult. With additional effort and perhaps a more appropriate choice of discrete modulus, it seems very likely that the tools presented in this paper place bounds on higher dimensional cases well within reach.
	
\section{Acknowledgements}
We thank Brad Lackey for useful discussions. This work was supported in part by the Joint Center for Quantum Information and Computer Science (QuICS), a collaboration between the University of Maryland Institute for Advanced Computer Studies (UMIACS) and the NIST Information Technology Laboratory (ITL). Portions of this paper are a contribution of NIST, an agency of the US government, and are not subject to US copyright.

\bibliographystyle{elsarticle-num-names}
\biboptions{sort&compress}

\end{document}